\documentclass[11pt,letterpaper]{article}
\usepackage[T1]{fontenc}
\usepackage[utf8]{inputenc}
\usepackage{lmodern}
\usepackage{authblk}
\usepackage{amssymb,amsfonts}
\usepackage[all,arc]{xy}
\usepackage{enumerate}
\usepackage{mathrsfs}
\usepackage{amsthm}
\usepackage{amsmath} 
\usepackage{tikz}
\usepackage{mathtools}
\usepackage{tikz}
\usepackage{float}
\usepackage{bbm}
\usepackage{hyperref}
\hypersetup{
colorlinks=true,
}

\newtheorem{thm}{Theorem}[section]

\newtheorem{defn}[thm]{Definition}

\newtheorem{exmp}[thm]{Example}

\newtheorem{rem}[thm]{Remark}

\makeatletter
\let\c@equation\c@thm
\makeatother

\title{Derivative Type Mapping Theorem for the Interpolative Berinde Weak Contraction in Metric Spaces with Application}
\author[1]{Clement Boateng Ampadu\thanks{profampadu@gmail.com}}
\affil[1]{Independent Researcher, USA}

\usepackage{hyperref}% http://ctan.org/pkg/hyperref
\AtBeginDocument{%
  \let\oldref\ref% 
  \def\ref{\oldref*}}

 \showoutput
\begin{document}

  \maketitle

\begin{abstract}
\noindent  Olatinwo \cite{3} introduced contractive definitions of the derivative type, and gave a new characterization of the Banach contraction principle and fixed point theorems for contractions defined implicitly. On the other hand Ampadu et.al \cite{4} introduced derivative type contractions in the setting of multiplictive metric spaces. In this paper, we have obtained a fixed point theorem of the derivative type for interpolative Berinde weak contractive mappings\cite{2} in the setting of metric spaces. An example is given to illustrate the main result of the paper. Finally, we apply our result to the Fredholm integral equation.
\end{abstract}

\begin{flushleft}
\textbf {AMS SUBJECT CLASSIFICATIONS:}  47H10, 54H25  
\end{flushleft} 
\begin{flushleft}
\end{flushleft}

               \hypersetup{linkcolor=blue}
             \tableofcontents

\section{\textbf{Preliminaries}}

\begin{thm} \cite{1} Let $(X,d)$ be a metric space, and suppose $T:X \mapsto X$ is a mapping satisfying the following contractive condition
$$d(Tx,Ty)\leq k d(x,y)$$ for all $x,y\in X$ and $k\in [0,1)$. If $(X,d)$ is complete, then $T$ has a unique fixed point.
\end{thm}

\begin{defn} \cite{2} Let $(X,d)$ be a metric space. We say $T:X\mapsto X$ is an interpolative Berinde weak operator if it satisfies
$$d(Tx,Ty)\leq \lambda d(x,y)^{\alpha} d(x,Tx)^{1-\alpha}$$ where $\lambda \in [0,1)$ and $\alpha \in (0,1)$ for all $x,y\in X\backslash Fix(T)$.
\end{defn}

\begin{defn} \cite{3} Let $(X,d)$ be a metric space. $X$ is said to be $d$-bounded if $$\delta_{d}(X)=\sup \{d(x,y)| x,y\in X\}<\infty$$
\end{defn}

\begin{defn} \cite{3} Let $(X,d)$ be a metric space. A map $T:X\mapsto X$ is said to be a  contraction of the derivative type if there exists $k\in [0,1)$ such that for every $x,y\in X$, we have $$\frac{d\varphi}{dt}\bigg|_{t=d(Tx,Ty)} \leq k \frac{d\varphi}{dt}\bigg|_{t=d(x,y)}$$ where $\varphi:\mathbb{R}^{+}\mapsto \mathbb{R}^{+}$ is a function such that
$\frac{d\varphi}{dt}\bigg|_{t=\epsilon} >0$ for each $\epsilon>0$.
\end{defn}

\begin{thm}\cite{3} Let $(X,d)$ be a complete metric space. and $T:X\mapsto X$ be a contraction of the derivative type.. Suppose that $X$ is $d$-bounded. Let
$\varphi:\mathbb{R}^{+}\mapsto \mathbb{R}^{+}$ be a function such that $\frac{d\varphi}{dt}\bigg|_{t=\epsilon} >0$ for each $\epsilon>0$. Then $T$ has a unique fixed point $z\in X$ such that for each $x\in X$. $\lim_{n\to\infty} T^{n}(x)=z$
\end{thm}

\section{\textbf{Main Result}}

\begin{defn} Let $(X,d)$ be a metric space. A map $T:X\mapsto X$ is called an interpolative Berinde weak operator of the derivative  type if there exist $\lambda \in [0,1)$ and $\alpha \in (0,1)$ such that
$$\frac{d\varphi}{dt}\bigg|_{t=d(Tx,Ty)} \leq \lambda \bigg[\frac{d\varphi}{dt}\bigg|_{t=d(x,y)}\bigg]^{\alpha}  \bigg[\frac{d\varphi}{dt}\bigg|_{t=d(x,Tx)}\bigg]^{1-\alpha}$$ for all $x,y\in X\backslash Fix(T)$, where $\varphi: \mathbb{R}^{+}\mapsto \mathbb{R}^{+}$ is a function such that $\frac{d \varphi}{dt}\bigg|_{t=\epsilon} >0$ for each $\epsilon>0$
\end{defn}

\begin{rem} If in Definition 2.1, we have
$$\frac{d\varphi}{dt}\bigg|_{t=d(Tx,Ty)}=\frac{d\varphi}{dt}\bigg|_{t=d(x,Tx)}=\frac{d\varphi}{dt}\bigg|_{t=d(x,y)}=t$$  then Definition 2.1 reduces to Definition 1.2
\end{rem}

\begin{thm} Let $(X,d)$ be a metric space, and $T:X\mapsto X$ be an interpolative Berinde weak operator of the derivative type. Suppose that $X$ is $d$-bounded. Let $\varphi: \mathbb{R}^{+}\mapsto \mathbb{R}^{+}$ be a function such that $\frac{d\varphi}{dt}\bigg|_{t=\epsilon} >0$ for each $\epsilon>0$. Then $T$ has a unique fixed point $x^{*}\in X$, such that for each $x\in X$, $\lim_{n\to\infty} T^{n}x=x^{*}$
\end{thm}

\begin{proof} Let $x_0\in X$ be chosen arbitrarily. Define the sequence $\{x_n\}_{n=0}^{\infty}$ recursively by $x_{n+1}=Tx_n$ for all $n$. Letting $x=x_{n-1}$ and $y=x_n$ in the contractive definition of the theorem we have
\begin{align*}
\frac{d\varphi}{dt}\bigg|_{t=d(x_n,x_{n+1})} & \leq \lambda \bigg[\frac{d\varphi}{dt}\bigg|_{t=d(x_{n-1},x_n)}\bigg]^{\alpha} \bigg[\frac{d\varphi}{dt}\bigg|_{t=d(x_{n-1},x_n)}\bigg]^{1-\alpha} \\
&= \lambda \frac{d\varphi}{dt}\bigg|_{t=d(x_{n-1},x_n)}
\end{align*}
By induction we obtain
\begin{align*}
\frac{d\varphi}{dt}\bigg|_{t=d(x_n,x_{n+1})} & \leq \lambda^{n} \frac{d\varphi}{dt}\bigg|_{t=d(x_0,x_1)} \\
& \leq \lambda^{n} \frac{d\varphi}{dt}\bigg|_{t=\delta_d(X)}
\end{align*}
where $d(x_0,x_1)\leq \delta_d(X)$, and $\delta_d(X)$ is stated as in Definition 1.3. Thus making use of the above inequality and the triangle inequality we have for all $n\geq 0$ and $m\geq 1$ that 
\begin{align*}
\frac{d\varphi}{dt}\bigg|_{t=d(x_n,x_{n+m})}&\leq \frac{d\varphi}{dt}\bigg|_{t=d(x_n,x_{n+1})}+\frac{d\varphi}{dt}\bigg|_{t=d(x_{n+1},x_{n+2})}+\cdots+ \frac{d\varphi}{dt}\bigg|_{t=d(x_{n+m-1},x_{n+m})} \\
&\leq \frac{d\varphi}{dt}\bigg|_{t=\delta_d(X)} (\lambda^{n}+\lambda^{n+1}+\cdots+\lambda^{n+m-1})\\
&=\frac{d\varphi}{dt}\bigg|_{t=\delta_d(X)} \lambda ^{n} \frac{1-\lambda^{m}}{1-\lambda}\\
&\leq \frac{d\varphi}{dt}\bigg|_{t=\delta_d(X)} \frac{\lambda^{n}}{1-\lambda}~\rightarrow 0~as~n,m \rightarrow \infty
\end{align*}
Since $\frac{d\varphi}{dt}\bigg|_{t=d(x_n,x_{n+m})} \rightarrow 0$, by the condition on $\varphi$, we have $d(x_n,x_{n+m})\rightarrow 0$. This shows that
$\{x_n\}$ is a Cauchy sequence in the complete metric space $(X,d)$, ensuring its convergence to a point $x^{*}\in X$. To confirm that $x^{*}$ is a fixed point, we
substitute $x=x^{*}$ and $y=x_n$ in the contractive definition of the theorem, then we have
$$ \frac{d\varphi}{dt}\bigg|_{t=d(Tx^{*},x_{n+1})} \leq \lambda \bigg[\frac{d\varphi}{dt}\bigg|_{t=d(x^{*},x_n)}\bigg]^{\alpha} \bigg[\frac{d\varphi}{dt}\bigg|_{t=d(x^{*},Tx^{*})}\bigg]^{1-\alpha}$$ Letting $n\rightarrow \infty$ in the above inequality, we have
$$\frac{d\varphi}{dt}\bigg|_{t=d(Tx^{*},x^{*})}\leq \lambda \bigg[\frac{d\varphi}{dt}\bigg|_{t=0}\bigg]^{\alpha} \bigg[\frac{d\varphi}{dt}\bigg|_{t=d(x^{*},Tx^{*})}\bigg]^{1-\alpha}$$ The condition on $\varphi$ gives $\frac{d\varphi}{dt}\bigg|_{t=0}=0$. So from the above inequality we have
$$\frac{d\varphi}{dt}\bigg|_{t=d(Tx^{*},x^{*})} \leq 0$$ which is a contradiction. Therefore by the condition on $\varphi$ we obtain
$$\frac{d\varphi}{dt}\bigg|_{t=d(Tx^{*},x^{*})}=0$$ this leads to $d(Tx^{*},x^{*})=0$ or $Tx^{*}=x^{*}$. We now prove that $T$ has a unique fixed point. Suppose this is not true, then there exists $x=Tx$, $y=Ty$, $x\neq y$, $d(x,y)>0$. Therefore using the contractive definition of the theorem, we have
\begin{align*}
\frac{d\varphi}{dt}\bigg|_{t=d(x,y)} &= \frac{d\varphi}{dt}\bigg|_{t=d(Tx,Ty)}\\
&\leq \lambda \bigg[\frac{d\varphi}{dt}\bigg|_{t=d(x,y)}\bigg]^{\alpha} \bigg[\frac{d\varphi}{dt}\bigg|_{t=0}\bigg]^{1-\alpha}
\end{align*}
The condition on $\varphi$ gives, $\frac{d\varphi}{dt}\bigg|_{t=0}=0$. Thus from the above inequality we have $$\frac{d\varphi}{dt}\bigg|_{t=d(x,y)}\leq 0$$ a contradiction. Again the condition on $\varphi$ gives, $\frac{d\varphi}{dt}\bigg|_{t=d(x,y)}=0$, which leads to $d(x,y)=0$ or $x=y$. This completes the proof.
\end{proof}

\begin{exmp} Let $X=[0,1]$, and define $d:X\times X\mapsto \mathbb{R}^{+}$ by, $d(x,y)=\vert x-y \vert$, for all $x,y\in X$.. Let
$T:X\mapsto X$ be defined by, $T(x)=\frac{x^{2}}{3}$, for all $x\in X$, also define $\varphi:\mathbb{R}^{+}\mapsto \mathbb{R}^{+}$ by,
$\varphi(x)= x^{3}$, for all $x\in \mathbb{R}^{+}$, then all the conditions of the above theorem  are satisfied and $x=0$ is the unique fixed point. Moreover, with
$\lambda=\alpha=\frac{1}{2}$ we get the following

\begin{figure}[H]
\centering
\includegraphics[width=\columnwidth]{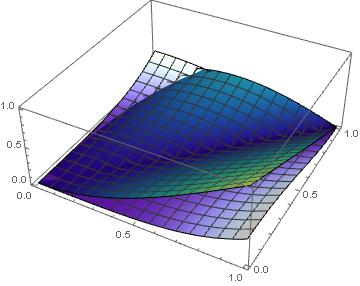}
\caption{The graph of $\frac{d\varphi}{dt}\bigg|_{t=d(Tx,Ty)}\leq \lambda \bigg[\frac{d\varphi}{dt}\bigg|_{t=d(x,y)}\bigg]^{\alpha} \bigg[\frac{d\varphi}{dt}\bigg|_{t=d(x,Tx)}\bigg]^{1-\alpha}$ for all $x,y\in (0,1]$ with left hand side being in ``LakeColors'' (bottom) and right hand side being in ``BlueGreenYellow'' (top)}
\end{figure}
\end{exmp}

\section{\textbf{Application}}

\noindent We apply our result to establish an existence theorem for non-linear  Fredholm integral equation. Let $Y=C[0,1]$ be a set of all real continuous functions on $[0,1]$ equipped with the metric $p(u,v)=\vert u-v\vert=\max_{t\in [0,1]} \vert u(t)-v(t)\vert$, for all $u,v\in C[0,1]$. Then $(Y,p)$ is a complete metric space. Now we consider the non-linear Fredholm integral equation
\begin{equation}
\tag{3.1}
u(t)=v(t)+\int_{0}^{1} K(t,s, u(s)) ds
\end{equation}
where $s,t\in [0,1]$. Assume that $K:[0,1]\times [0,1]\times Y: \mapsto \mathbb{R}$ and $v:[0,1]\mapsto \mathbb{R}$ are continuous, where $v(t)$ is a given function in $Y$

\begin{thm} Suppose $(Y,p)$ is a metric space equipped with the metric $p(u,v)=\vert u-v\vert=\max_{t\in[0,1]} \vert u(t)-v(t)\vert$ for all $u,v\in Y$, and $F:Y\mapsto Y$ be an operator on $Y$ defined by
\begin{equation}
\tag{3.2}
Fu(t)=v(t)+\int_{0}^{1} K(t,s,u(s)) ds
\end{equation}
If there exists $\lambda \in [0,1)$ and $\alpha \in (0,1)$ such that for all $u,v\in Y$, $s,t\in [0,1]$, satisfying the following inequality
$$\frac{d\varphi}{dq}\bigg|_{q=\vert K(t,s,u(s))-K(t,s,v(s))\vert} \leq \lambda \bigg[\frac{d\varphi}{dq}\bigg|_{q=\vert u(s)-v(s)\vert}\bigg]^{\alpha} \bigg[\frac{d\varphi}{dq}\bigg|_{q=\vert u(s)-Fu(s)\vert}\bigg]^{1-\alpha}$$ where $\varphi:\mathbb{R}^{+}\mapsto \mathbb{R}^{+}$ is a function defined as in Theorem 2.3. Then the integral equation (3.2) has a unique solution in $Y$
\end{thm}

\begin{proof} From (3.1) and (3.2) we obtain
\begin{align*}
\frac{d\varphi}{dq}\bigg|_{q=\vert Fu(t)-Fv(t)\vert} &=\frac{d\varphi}{dq}\bigg|_{q=\bigg\vert \int_{0}^{1} K(t,s,u(s)) ds-\int_{0}^{1} K(t,s,v(s))ds \bigg\vert} \\
&\leq \frac{d\varphi}{dq}\bigg|_{q=\int_{0}^{1} \vert K(t,s,u(s))-K(t,s,v(s))\vert ds} \\
&\leq \int_{0}^{1} \frac{d\varphi}{dq}\bigg|_{q=\vert K(t,s,u(s))-K(t,s,v(s))\vert} ds\\
&\leq \int_{0}^{1} \bigg\{\lambda \bigg[\frac{d\varphi}{dq}\bigg|_{q=\vert u(s)-v(s)\vert}\bigg]^{\alpha} \bigg[\frac{d\varphi}{dq}\bigg|_{q=\vert u(s)-Fu(s)\vert}\bigg]^{1-\alpha} \bigg\} ds
\end{align*}
\noindent Taking maximum on both sides for all $t\in [0,1]$, we conclude that
$$\frac{d\varphi}{dq}\bigg|_{q=p(Fu,Fv)} \leq \lambda \bigg[\frac{d\varphi}{dq}\bigg|_{q=p(u,v)}\bigg]^{\alpha} \bigg[\frac{d\varphi}{dq}\bigg|_{q=p(u,Fu)}\bigg]^{1-\alpha} $$
Since $Y=C[0,1]$ is a complete metric space, therefore all the conditions of Theorem 2.3 are satisfied, and hence the integral equation (3.2) has a unique solution in $Y$.
\end{proof}

\end{document}